\documentclass[a4paper]{amsart}
\usepackage{amssymb}
\usepackage{amsmath} 
\usepackage{pifont}
\usepackage[mathb]{mathabx}
\usepackage[usenames]{color}
\usepackage{amsfonts}
\usepackage[all]{xypic}
\usepackage{graphicx}
\usepackage{psfrag}
\usepackage{verbatim}
\usepackage{pstricks}
\usepackage[latin1]{inputenc}
\usepackage{mathrsfs}
\usepackage[all,knot]{xy}
\usepackage{amssymb, amsthm, amsmath, pifont} 
\usepackage{enumerate}
\usepackage{pdfpages}

\newtheorem{thm}{Theorem}
\newtheorem{corollary}[thm]{Corollary}

\newtheorem{example}[thm]{Example}
\newtheorem{definition}[thm]{Definition}
\newtheorem{remark}[thm]{Remark}
\newtheorem{prop}[thm]{Proposition}
\newtheorem{claim}[thm]{Claim}
\newtheorem{observation}[thm]{Observation}

\newtheorem{question}[thm]{question}

\newenvironment{Example}{\begin{example}\rm}{\end{example}}

\newenvironment{Remark}{\begin{remark}\rm}{\end{remark}}

\def\et{\quad\mbox{and}\quad}

\def\R{\mathbb{R}}

\def\epsilon{\varepsilon}

\def\s{{\sigma}}
\def\b{{\rm{b_1}}}
\def\l{{l}}
\begin{document}
\author{Peter Feller}
\email{peter.feller@math.ch}
\thanks{The author gratefully acknowledges support by the Swiss National Science Foundation Grant~155477.}
\keywords{Positive braids, signature, Goeritz form, Gordon-Litherland pairing, topological slice genus}
\subjclass[2010]{57M25,  57M27}
\address{ETH Z{\"u}rich, R{\"a}mistrasse 101, 8092 Z{\"u}rich, Switzerland}
\title{A sharp signature bound for positive four-braids}
\begin{abstract}
We provide the optimal linear bound for the signature of positive four-braids in terms of the three-genus of their closures.
As a consequence, we improve previously known linear bounds for the signature in terms of the first Betti number for all positive braid links.
We obtain our results by combining bounds for positive three-braids with Gordon and Litherland's approach to signature via unoriented surfaces and their Goeritz forms.
Examples of families of positive four-braids for which the bounds are sharp are provided.

\end{abstract}
\maketitle

\section{Introduction}
This paper is concerned with \emph{positive braid knots and links}\textemdash the knots and links obtained as the closure $\widehat{\beta}$ of positive braids $\beta$\textemdash and the following two link invariants: the \emph{first Betti number} $\b(L)$ of a link $L$\textemdash the minimal first Betti number of oriented surfaces in $\R^3$ with oriented boundary $L$\textemdash and the \emph{signature} $\s(L)$ of a link $L$ as introduced by Trotter~\cite{Trotter_62_HomologywithApptoKnotTheory} and (for links) by Murasugi~\cite{Murasugi_OnACertainNumericalInvariant}. Let a link $L$ be the closure of a non-trivial positive braid\textemdash a positive braid such that its closure is not an unlink. We suspect (as conjectured in~\cite{Feller_14_Sigofposbraids}) that 
\begin{equation}\label{eq:1/2conj}\b(L)\geq{-\s}(L)
>\frac{\b(L)}{2},\end{equation}
where the first inequality is immediate from the definition of the signature.
This article establishes~~\eqref{eq:1/2conj} for closures of positive $4$--braids. 
\begin{thm}\label{thm:1/2for4braids}
Let $\beta$ be a positive $4$--braid such that its closure $\widehat{\beta}$ is not an unlink, then
\[{-\s}(\widehat{\beta})
>\frac{\b(\widehat{\beta})}{2}.\]
\end{thm}
By an observation~\cite[Reduction Lemma]{Feller_14_Sigofposbraids}, which we recall in the appendix for the reader's convenience, Theorem~\ref{thm:1/2for4braids} implies the following 
bound for all positive braids:

\begin{thm}\label{thm:1/8}Let $\beta$ be a positive braid such that its closure $\widehat{\beta}$ is not an unlink, then
\[{-\s}(\widehat{\beta})
>\frac{\b(\widehat{\beta})}{8}.\]
\end{thm}
In other words, up to a factor of $4$,~\eqref{eq:1/2conj} holds. 
The main technical ingredient in the proof of Theorem~\ref{thm:1/2for4braids} is Gordon and Litherland's approach of using (non-oriented) checkerboard surfaces and the associated Goeritz form to calculate the signature~\cite{GordonLitherland_78_OnTheSigOfALink}.

Let us shortly put Theorems~\ref{thm:1/2for4braids} and~\ref{thm:1/8} in context.
Links that arise as closures of positive braids are a well-studied class of links containing important families 
such as (positive) torus links, algebraic links, and Lorenz links, while themselves being a subclass of positive links.
Rudolph established that closures of positive braids have strictly negative signature~\cite{Rudolph_83_PosBraidPosSig}.
For positive $4$--braids, previous results by Stoimenow~\cite[Theorem~4.2]{Stoimenow_08} and the author~\cite[Main Proposition]{Feller_14_Sigofposbraids},
provided linear bounds of the signature in terms of the first Betti number with factor $\frac{2}{11}$ and $\frac{5}{12}$, respectively.
For the more general class of positive links, Baader, Dehornoy, and Liechti 
provide a linear bound for the signature in terms of the first Betti number with factor $\frac{1}{48}$~\cite[Theorem~2]{BaaderDehornoyLiechti_15_SigandConofPosKnots}.
The novelty of Theorem~\ref{thm:1/2for4braids} is that the factor $\frac{1}{2}$ of the linear bound is optimal; compare the discussion in Section~\ref{sec:examples}. The linear bound with  factor $\frac{1}{8}$ for all positive braids as provided in Theorem~\ref{thm:1/8} is a clear improvement over the best known previous bound; compare~\cite{Feller_14_Sigofposbraids} and~\cite[Theorem~2]{BaaderDehornoyLiechti_15_SigandConofPosKnots}.
A geometric consequence of the signature bounds are lower bounds for the topological slice genus $g_{4}^{top}(K)$ of a knot $K$\textemdash the minimal genus among all locally flat oriented surfaces in the unit $4$-ball $B^4$ with boundary $K\subset S^3=\partial B^4$.
Indeed, combining Theorem~\ref{thm:1/2for4braids} and Theorem~\ref{thm:1/8} with Kauffman and Taylor's result that
$|\s(K)|\leq 2g_{4}^{top}(K)$ for all knots $K$~\cite{KauffmanTaylor_76_SignatureOfLinks}
yields the following.
For a knot $K$, denote the three-genus\textemdash half the first Betti number of $K$\textemdash by $g(K)$.
\begin{corollary}\label{cor:topgenusbounds}
For a knot $K$ that is not the unknot and that is the closure of a 
positive braid, one has
\[g(K)\geq g_{4}^{top}(K)>\frac{g(K)}{8}.\] Furthermore, if $K$ is the closure of a positive $4$--braid, then
\[g(K)\geq g_{4}^{top}(K)>\frac{g(K)}{2}.\]
\end{corollary}
\begin{Remark}
There are knots $K$ that are closures of positive braids for which
$g(K)$ is strictly larger than $g_{4}^{top}(K)$, which is surprising since in the smooth setting the smooth slice genus is equal to the three-genus
by Kronheimer and Mrowka's resolution of the Thom conjecture~\cite[Corollary~1.3]{KronheimerMrowka_Gaugetheoryforemb}.
Indeed, Rudolph observed that the topological slice genus of the torus knot $T_{5,6}$ is strictly less than $10=g(T_{5,6})$.
In fact, there are infinite families of positive braid knots for which the topological slice genus can be linearly bounded away from the three-genus by a significant amount~\cite[Theorem~2]{Rudolph_84_SomeTopLocFlatSurf},~\cite{BaaderFellerLewarkLiechti_15}.
This justifies interest in a linear lower bound as provided in Corollary~\ref{cor:topgenusbounds}.\end{Remark}

We conclude the introduction by outlining the strategy of the proof of Theorem~\ref{thm:1/2for4braids}, as provided in Section~\ref{sec:1/2for4braids}. We will use Gordon and Litherland's approach to the signature via Goeritz forms to show the following. For every link $L$ that is the closure of a non-trivial positive $4$--braid, there exists a link $L'$ that is the closure of a positive $3$--braid with
\[\b(L')=\b(L)+1\et |\s(L)- \s(L')|\leq 1.\] This will allow us to reduce Theorem~\ref{thm:1/2for4braids} to the following proposition.
\begin{prop}\label{prop:1/2for3braids}
Let $\beta$ be a 
positive $3$--braid with $\b(\widehat{\beta})\geq 2$,
then
\[{-\s}(\widehat{\beta})\geq \frac{\b(\widehat{\beta})}{2}+1.\]
\end{prop}
Proposition~\ref{prop:1/2for3braids} improves Stoimenow's result that
${-\s}(\widehat{\beta})
>\frac{\b(\widehat{\beta})}{2}$ {for positive non-trivial $3$--braids~\cite[Theorem~4.1]{Stoimenow_08}}.
We provide a proof for Proposition~\ref{prop:1/2for3braids} which is independent of Stoimenow's techniques; see Section~\ref{sec:1/2+1for3braids}. 
 Optimality of Theorem~\ref{thm:1/2for4braids} and Proposition~\ref{prop:1/2for3braids} is discussed in Section~\ref{sec:examples}.

 {\bf{Acknowledgements}:} I thank Josh Greene for sharing his work on checkerboard surfaces
 , which led me to consider to use them to prove Theorem~\ref{thm:1/2for4braids}. Thanks also to Sebastian Baader for helpful remarks. I thank both of them for the fun hours we spent calculating signatures. 

\section{Setup: Signatures of links via Goeritz forms and positive braids}\label{sec:notions}
We set up notions and recall facts about braids and the signature of links.
\subsection{Signature of links and Goeritz forms}
For a link $L$\textemdash an oriented smooth embedding of a non-empty finite union of circles in $S^3$\textemdash the \emph{signature}, denoted by $\s(L)$, is defined to be the signature of the symmetrized Seifert form on $H_1(F)$, where $F$ is any 
compact and oriented surface in $S^3$ with oriented boundary $L$; compare Trotter and Murasugi~\cite{Trotter_62_HomologywithApptoKnotTheory,Murasugi_OnACertainNumericalInvariant}. In particular,
one has that $-\b\leq\s\leq\b$ holds for all links. Unifying Trotter's approach to the signature and work of Goeritz~\cite{Goeritz_33},
Gordon and Litherland~\cite{GordonLitherland_78_OnTheSigOfALink} 
introduced the following procedure to calculate the signature.
For any link diagram $D_L$ of $L$---the image of a generic project of the link $L$ to a standard $2$--sphere $\R^2\cup\{\infty\}$ in $S^3$ together with crossing information---
one has
\begin{equation*}
\s(L)=\s(S_L)-\mu(S_L),
\end{equation*}
where $S_L$ is a non-oriented surface with boundary $L$ given as one of the two checkerboard surfaces\footnote{$S_L$ is contained in $D_L\subset S^3$ away from neighborhoods of crossings. In a neighborhood of a crossing, $S_L$ is given by a small `half-twisted' band. 
We refer to Figure~\ref{fig:Sfor4braids} for an illustrative example and 
to Gordon and Litherland's original work~\cite{GordonLitherland_78_OnTheSigOfALink} for more details.} of $D_L$ and $\s(S_L)$ and $\mu(S_L)$ are defined as follows.
To every crossing $c$ of $D_L$ one associates a type (I or II) and a sign $\eta(c)$ (1 or -1) by the rule specified in Figure~\ref{fig:typeandsign}.
\begin{figure}[h]
\centering
\def\svgscale{1.5}
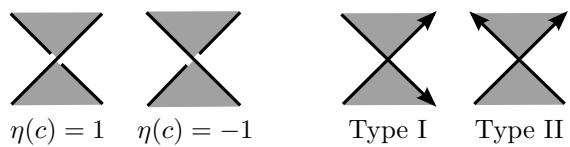
\caption{Left: Sign associated to a crossing. Right: Type associated to a crossing. The type only depends on whether or not the surface $S_L$ (gray) can locally be given an orientation that induces the orientation of the link. In particular, the type is independent of the crossing information.}
\label{fig:typeandsign}
\end{figure}
Then one defines \begin{equation}\label{eq:def:mu}\mu(S_L)=\sum_{c \in \text{crossings of }L\text{ of type II}}\eta(c).\end{equation}
To define $\s(S_L)$, pick a basis $[\delta_1],\cdots,[\delta_k]$ of $H_1(S_L)$ represented by simple closed curves $\delta_i\subset S_L$ and let the matrix 
$G_L=\{g_{ij}\}$ be 
given by
$g_{ij}=lk(\delta_i,\delta_j^{\pm})$. 
 Here $\delta_j^{\pm}$ denotes the link in
$S^3\backslash S_L$ obtained from $\delta_j\subset S_L$ by a small push-off in both normal directions of $S_L$ and $lk$ denotes the linking number in $S^3$. Then one sets
$\s(S_L)=\s(G_L)$,
where $\s(G_L)$ denotes the signature of $G_L$\textemdash the number of positive eigenvalues minus the number of negative eigenvalues counted with multiplicities. The bilinear form defined by $G_L$ is called the \emph{Goeritz form}. This fits into the setting of the more general \emph{Gordon-Litherland pairing}, where one uses any (in general non-orientable) surface $S_L$ with (unoriented) boundary $L$ rather than a checkerboard surface and $-\mu(S_L)$ is replaced by half the Euler number of $S_L$; see~\cite[Corollary~5'']{GordonLitherland_78_OnTheSigOfALink}.
A warning concerning sign conventions is in order: the above definition of $\s(L)$ has opposite sign of that given in~\cite{Rudolph_83_PosBraidPosSig,Stoimenow_08,BaaderDehornoyLiechti_15_SigandConofPosKnots}. The present convention agrees with the convention in~\cite{Trotter_62_HomologywithApptoKnotTheory,Murasugi_OnACertainNumericalInvariant,GordonLitherland_78_OnTheSigOfALink} and appears to be the standard one. For example, \begin{equation}\label{eq:sigbraidindex2}
\s(T_{2,{n+1}})=-n\quad\text{(rather than $n$)  for all positive integers $n$.}
                                                 \end{equation}

The following properties of the signature of a link follow rather directly from both the original definition and Gordon and Litherland's approach.
If a link $L'$ can be obtained from a link $L$ by one saddle move, then
\begin{equation}\label{eq:sigundersaddlemove}|\s(L)-\s(L')|\leq 1\quad\text{\cite[Lemma~7.1]{Murasugi_OnACertainNumericalInvariant}}.\end{equation}
Here a \emph{saddle move} is defined as changing the link in a $3$--ball as described on the left-hand side in Figure~\ref{fig:saddlemove}.
\begin{figure}[h]
\centering
\def\svgscale{1.9}
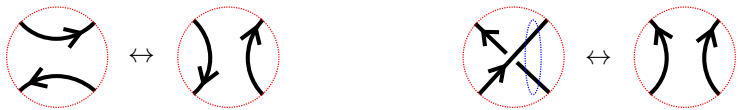
\caption{Left: A saddle move. Right: A smoothing of a crossing, which is obtained by applying a saddle move to the blue sphere}
\label{fig:saddlemove}
\end{figure}
If $L'$ can be obtained from $L$ by \emph{smoothing of a crossing}, defined on the right-hand side in Figure~\ref{fig:saddlemove},
then~\eqref{eq:sigundersaddlemove} also holds since a smoothing of a crossing corresponds to a saddle move on the link.
The signature is additive under split union and connected sum: {for all links $L$ and $L'$}
\begin{equation}\label{eq:sigisadd}\s(L\sqcup L')=\s(L\sharp L')=\s(L)+\s(L')
\quad\text{\cite[Lemma~7.2 and Lemma~7.3]{Murasugi_OnACertainNumericalInvariant}}.
\end{equation}
{Note that the connected sum $L\sharp L'$ is only well-defined if one specifies which components of the two links $L$ and $L'$ are connected via the connect
sum operation; however,~\eqref{eq:sigisadd} holds for all such specifications.}
\subsection{Positive braids}\label{sec:posbr}
Let $B_b$ be Artin's \emph{braid group} on $b$ strands~\cite{Artin_TheorieDerZoepfe}.
\[B_b=\left<a_1,\cdots,a_{b-1}\;|\; a_{i} a_{j}= a_{j} a_{i} \text{ for }|i-j|\geq 2,\;  a_{i} a_{j} a_{i}= a_{j} a_{i} a_{j} \text{ for } |i-j|=1\right>.\]
A braid $\beta$ in $B_b$ corresponds to a (geometric) braid, represented via braid diagrams, and yields a link $\widehat{\beta}$ via the closure operation. Generators $a_i$ (respectively $a_i^{-1}$) correspond to braids given by the braid diagram where the $i$th and $(i+1)$-st strands cross once positively (respectively negatively). This yields that a \emph{braid word}---a word in the generators of $B_b$---defines a diagram for the braid it encodes.
We illustrate this in Figure~\ref{fig:braidandclosure} for one example to set conventions and refer to~\cite{Birman_74_BraidsAMSStudies} for a detailed account. 
\begin{figure}[h]
\centering
\def\svgscale{0.7}
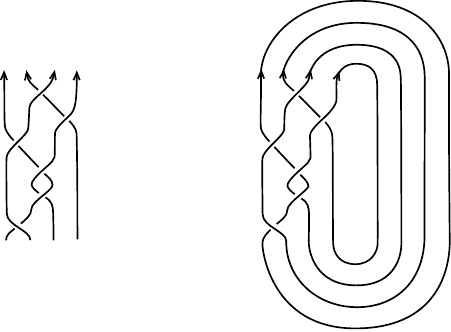
\caption{A braid diagram (left) coming from the braid word $a_1^{-1}a_2a_2a_1a_3a_2$ and a link diagram (right) for the corresponding braid closure.
}
\label{fig:braidandclosure}
\end{figure}

A \emph{positive braid} on $b$ strands is an element $\beta$ in $B_b$ that can be written as positive braid word $ a_{s_{1}} a_{s_{2}}\cdots a_{s_{\l(\beta)}}$ with $s_{i}\in\{1,\cdots,b-1\}$, where
$\l(\beta)$ is called the \emph{length} or \emph{writhe} of $\beta$. Note that $\l(\beta)$ an invariant of positive braids $\beta$ since it is independent of the choice of a positive braid word for $\beta$.
The first Betti number of the closure of positive braids is understood as a consequence of Bennequin's inequality~\cite[Theorem~3]{Bennequin_Entrelacements}, which implies the following formula:
\begin{equation}\label{eq:b}
\b(\widehat{\beta})=\l({\beta})-b+c\;\;\text{for every positive braid }\beta,
\end{equation}
where $b$ is the number of strands of $\beta$ and $c$
equals $1$ plus the number of generators $ a_{i}$ that are not used in a positive braid word for $\beta$.
For example for all positive integers $n$ and $m$, the torus link $T_{n,m}$ is a positive braid link with first Betti number
$(n-1)(m-1)$ since the closure of the braid $(a_1\cdots a_{n-1})^m\in B_n$ is $T_{n,m}$.

\section{Proof of Theorem~\ref{thm:1/2for4braids}}\label{sec:1/2for4braids}
In this section we prove that, if a link $L$ with $\b(L)\geq1$ is the closure of a positive $4$--braid, then
 \[{-\s}(L)\geq \frac{\b(L)}{2}+\frac{1}{2} >\frac{\b(L)}{2}.\] Besides the notions introduced in Section~\ref{sec:notions}, the proof uses Proposition~\ref{prop:1/2for3braids}, which is proved in Section~\ref{sec:1/2+1for3braids}.

\begin{proof}[Proof of Theorem~\ref{thm:1/2for4braids}]
For the entire proof, let $L$ be the closure of a non-trivial positive $4$--braid given by a positive $4$--braid word $\beta$. Without loss of generality all generators $a_1$, $a_2$, and $a_3$ appear in $\beta$ at least twice; in particular, $b_1(L)\geq 3$ by~\eqref{eq:b}.
Indeed, otherwise $L$ is a connected sum or split union of links that are closures of positive braids on $3$ or less strands, for which the statement follows from Proposition~\ref{prop:1/2for3braids}, \eqref{eq:sigbraidindex2}, and the fact that signature is additive on connect sums and split unions.

We define a $4$--braid word $\alpha$ by replacing all $a_3$ in $\beta$ by $a_1$. For example,
if $\beta=a_1a_3a_2a_1a_3a_2a_2a_1a_3a_2$, then $\alpha=a_1a_1a_2a_1a_1a_2a_2a_1a_1a_2$.

From the given braid word $\beta$ we get a braid diagram and a corresponding link diagram $D_\beta$ (compare Section~\ref{sec:posbr}) representing $L=\widehat{\beta}$. Similarly, the braid word $\alpha$ yields a link diagram $D_\alpha$ for $\widehat{\alpha}$. Next we checkerboard color the two diagrams $D_\beta$ and $D_\alpha$ (we choose the convention that the unbounded region of a diagram is white) and we denote the black checkerboard surfaces by $S_\beta$ and $S_\alpha$, respectively. See Figure~\ref{fig:Sfor4braids}, where this is illustrated for an example.
\begin{figure}[h]
\centering
\def\svgscale{1.6}
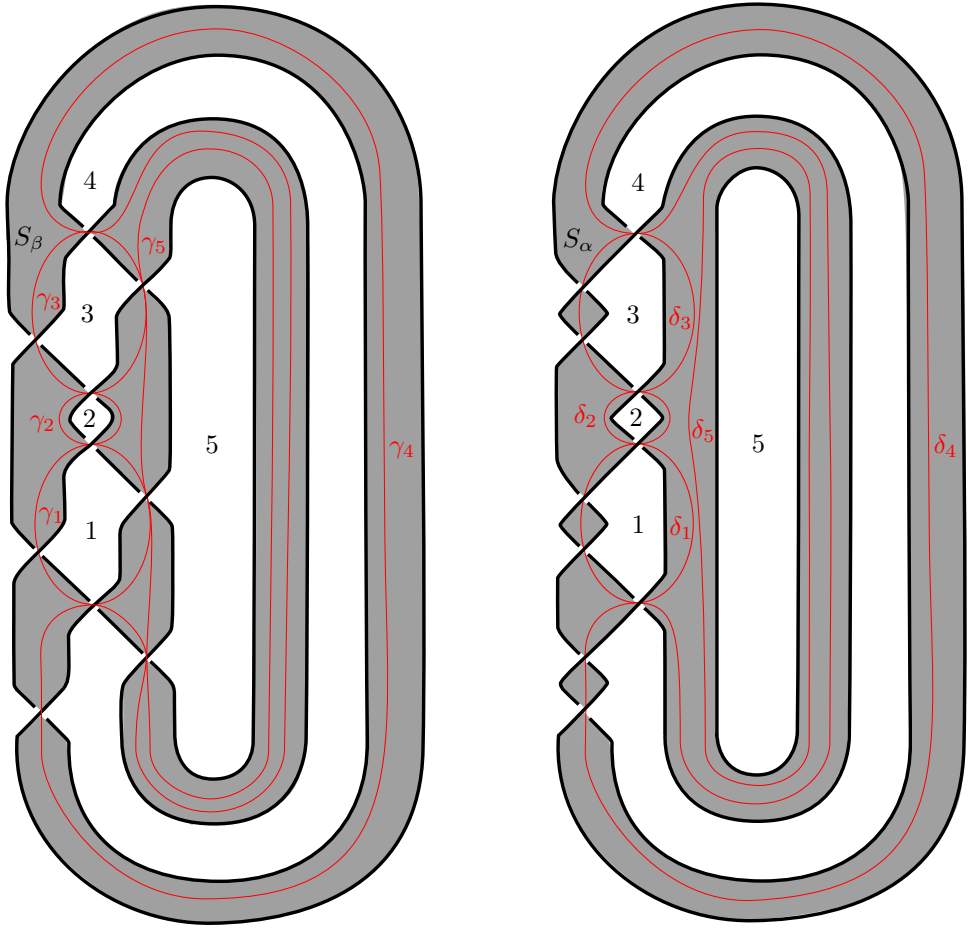
\caption{The checkerboard surfaces $S_\beta$ and $S_\alpha$ (gray) associated with the braid words $\beta=a_1a_3a_2a_1a_3a_2a_2a_1a_3a_2$ (left) and $\alpha=a_1a_1a_2a_1a_1a_2a_2a_1a_1a_2$ (right). The curves $\gamma_i$ and $\delta_i$ (red) constitute bases of $H_1(S_\beta)$ and $H_1(S_\alpha)$, respectively.}
\label{fig:Sfor4braids}
\end{figure}
We note that the diagrams $D_\beta$ and $D_\alpha$ and the checkerboard surfaces $S_\beta$ and $S_\alpha$ depend on the braid words $\beta$ and $\alpha$, respectively, (rather than just depending on the braids that $\beta$ and $\alpha$ represent in $B_4$).

Let $k$ denote the number of $a_2$ in the braid word $\beta$. The number of $a_2$ in the braid word $\alpha$ is also $k$. The checkerboard surfaces $S_\beta$ and $S_\alpha$ have first Betti number $k+1$.
This is best seen by observing that the Betti number of $S_\beta$ and $S_\alpha$ are equal to the number of bounded white regions in $D_\beta$ and $D_\alpha$, respectively, of which there are $k+1$. We describe this in a bit more detail and set up a one-to-one correspondence between the bounded white regions of $D_\beta$ and $D_\alpha$, which will be  useful below.
Every $a_2$ in $\beta$ corresponds to a crossing in $D_\beta$, which is touched by two white regions (one from above and one from below). We label the generators $a_2$ in $\beta$ by $1$, $2$, $\cdots$, $k$ in order of appearance from the left and assign the same label to the unique white region touching the corresponding crossings from above. Only one bounded white region remains. We label this bounded white region by $k+1$. 
The same labeling procedure applied to $\alpha$ yields a labeling for the white bounded regions of $D_\alpha$ and this sets up a one-to-one correspondence between the bounded white regions of $D_\beta$ and $D_\alpha$. See Figure~\ref{fig:Sfor4braids}, where this labeling is illustrated.

Following Goeritz and Gordon-Litherland~\cite{Goeritz_33,GordonLitherland_78_OnTheSigOfALink},
we choose bases $([\gamma_1],\cdots,[\gamma_{k}])$ and $([\delta_1],\cdots,[\delta_{k}])$ for $H_1(S_\beta)$ and $H_1(S_\alpha)$, respectively, as follows. 
The boundary of a white region in $D_\beta$ (respectively $D_\alpha$) labeled $i$ defines a curve $\gamma_i$ (respectively $\delta_i$) in $S_\beta$ (respectively $S_\alpha$). We orient the $\gamma_i$ and $\delta_i$ counterclockwise and let $[\gamma_i]$ and $[\delta_i]$ denote the corresponding homology classes in $H_1(S_\beta)$ and $H_1(S_\alpha)$, respectively.


Let $G_\beta$ and $G_\alpha$ denote the Goeritz matrices of $S_\beta$ and $S_\alpha$ with respect to the bases $([\gamma_1],\cdots,[\gamma_{k}])$ and $([\delta_1],\cdots,[\delta_k])$, respectively.
The above one-to-one correspondence between the white bounded regions of $D_\beta$ and $D_\alpha$ is set up such that for all $i,j\leq k$,
\[lk(\gamma_i,\gamma_j^{\pm})=
lk(\delta_i,\delta_j^{\pm}).\] 
In other words, all entries of $G_\alpha$ not in the last row or column coincide with the corresponding entries
of $G_\beta$. Since the signatures of two real-valued symmetric $(k+1)\times(k+1)$ matrices that are identical except in the last column and row differ by at most one, we have 
\begin{equation}\label{eq:sigMbMa}
{-\s}(G_\beta)\geq{-\s}(G_\alpha)-1.
\end{equation}
Furthermore, $\mu(S_\beta)=\mu(S_\alpha)$ since $\mu(S_\beta)$ (respectively $\mu(S_\alpha)$) is equal to the number of $a_1$ and $a_3$ in $\beta$ (respectively $\alpha$) by~\eqref{eq:def:mu}.

We note that $\widehat{\alpha}$ is the split union of the unknot and the link $L'$ obtained as the closure of the $3$--braid given by interpreting $\alpha$ as a $3$--braid word. Therefore, $\s(L')=\s(\widehat{\alpha})$ by~\eqref{eq:sigisadd}. Note also that $\b(L')=\b(L)+1$ by~\eqref{eq:b}.

With all of the above and Gordon and Litherland's
\[\s(L)=\s(G_\beta)-\mu(S_\beta)\et \s(\widehat{\alpha})=\s(G_\alpha)-\mu(S_\alpha),\]
we calculate
\begin{align*}
{-\s}(L)&=-\s(G_\beta)+\mu(S_\beta)\\
&\geq -\s(G_\alpha)-1+\mu(S_\beta)\\
&=-\s(G_\alpha)-1+\mu(S_\alpha)\\
&={-\s}(\widehat{\alpha})-1\\
&={-\s}(L')-1\\
&\geq\frac{\b(L')}{2}+1-1\\
&=\frac{\b(L)+1}{2},
\end{align*}
where we used \eqref{eq:sigMbMa}, $\mu(S_\beta)=\mu(S_\alpha)$, Proposition~\ref{prop:1/2for3braids}, and \eqref{eq:b} in
the second, third, sixed, and last line, respectively.
This concludes the proof of
\[{-\s}(L)\geq\frac{\b(L)+1}{2}>\frac{\b(L)}{2},\]
for all closures $L$ of non-trivial positive $4$--braids.
\end{proof}

\section{Proof of Proposition~\ref{prop:1/2for3braids}}\label{sec:1/2+1for3braids}
In this section we prove that, for all positive $3$--braids $\beta$ with $\b(\widehat{\beta})\geq2$, we have
\begin{equation}\label{eq:1/2+1for3braids}{-\s}(\widehat{\beta})\geq
\frac{\b(\widehat{\beta})}{2}+1.\end{equation}
\begin{proof}[Proof of Proposition~\ref{prop:1/2for3braids}]
 Let $\beta$ be a 
 positive $3$--braid with $\b(\widehat{\beta})\geq 2$. Denote by $\Delta$ the positive half-twist on $3$-strands $a_1a_2a_1=a_2a_1a_2$.
Let $n_\beta$ be the maximum of all non-negative integer $n'$ such that $\beta=\Delta^{n'}\alpha$ for a positive $3$--braid $\alpha$. Note that $n_\beta$ exists since the $n'$ are less or equal than a third of the length $l(\beta)$.
Since we are interested in the signature and the first Betti number of the closure of $\beta$, there is no loss of generality by assuming that $n_\beta$ is maximal among all such $n'$ for $\beta'=\Delta^{n'}\alpha'$, where $\beta'$ is any positive braid conjugate to $\beta$. This maximality of $n_\beta$ is assumed in the entire proof.
 For example, the braid $\beta=a_1a_2a_2a_2a_2$, which has $n_\beta=0$, is not considered
 since it is conjugate to $a_2a_1a_2a_2a_2=\Delta a_2a_2$. Furthermore, we assume in the entire proof that $\widehat{\beta}$ is not the closure of a $2$--braid since for such closures one has \[{-\s}(\widehat{\beta})\overset{\text{\eqref{eq:sigbraidindex2}}}{=}\l(\beta)-1\overset{\text{\eqref{eq:b}}}{=}\b(\widehat{\beta})\geq \frac{\b(\widehat{\beta})}{2}+1.\]
For example for all positive integers $l$, the braid $\beta=\Delta a_1^l$ is not considered since its closure is the braid index two torus knot $T_{2,l+2}$.

 If ${-\s}(\widehat{\beta})\geq\frac{\b(\widehat{\beta})}{2}$ holds for a positive $3$--braid $\beta$,
 then one also has
 \begin{align*}{-\s}\left(\widehat{(\Delta^4)^k\beta}\right)
 &\geq\frac{\b\left(\widehat{(\Delta^4)^k\beta}\right)}{2}+1,\end{align*}
 for all positive integers $k$. This can be seen as follows. 
 Gambaudo and Ghys established that
 \begin{equation}\label{eq:sigunderdoubletwist}{-\s}\left(\widehat{\left(\Delta^4\right)^k\beta}\right)
 ={-\s}(\widehat{\beta})+8k\quad\text{\cite[Lemma~4.1]{GambaudoGhys_BraidsSignatures}},\end{equation}
 for all $3$--braids $\beta$ and all integers $k$. Thus, ${-\s}(\widehat{\beta})\geq\frac{\b(\widehat{\beta})}{2}$ yields
 \begin{align*}
 {-\s}\left(\widehat{(\Delta^4)^k\beta}\right)&\overset{\text{\eqref{eq:sigunderdoubletwist}}}{=}{-\s}(\widehat{\beta})+8k\\
 &\geq \frac{\b(\widehat{\beta})}{2}+8k\\
 &\overset{\text{\eqref{eq:b}}}{\geq}\frac{\b\left(\widehat{(\Delta^4)^k\beta}\right)}{2}+2k\\
 &\geq\frac{\b\left(\widehat{(\Delta^4)^k\beta}\right)}{2}+1.
 \end{align*}
 Therefore, it suffices to establish~\eqref{eq:1/2+1for3braids}
 for positive $3$--braids $\beta$ with $\b(\widehat{\beta})\geq 2$ and $n_\beta=0,1,2$, or $3$.

 Let us first consider the case $n_\beta=0$.
 After possibly a conjugation (given by cyclic permutation of positive braid words), we have that $\beta$ is given by a positive braid word starting with $a_1$ and ending with $a_2$, i.e.
\begin{equation*}\label{eq:alphaforn=0}\beta=a_1^{l_0}a_2^{l_1}a_1^{l_2}\cdots a_1^{l_{c-1}}a_2^{l_{c}}\end{equation*}
for some odd integer $c\geq 1$ and positive integers $l_i$. Furthermore, we have that for all $i\in\{0,\cdots,c\}$ the $l_i$ are integers larger than or equal to $2$ since otherwise (up to cyclic permutation) $\beta$ will contain $a_1a_2a_1$ or $a_2a_1a_2$ as a subword which is impossible by the definition of $n_\beta$.
Note that $\b(\widehat{\beta})\geq 2c$ by~\eqref{eq:b}.
By applying $\frac{c-1}{2}$ saddle moves on $\widehat{\beta}$, 
we get a link $L$ that is a connected sum of $\frac{c+3}{2}$ torus links of braid index two; in fact, a connect sum of the torus link $T_{2,l_{even}}$, where $l_{even}=\sum_{i=0}^{\frac{c-1}{2}} l_{2i}$, and the torus links $T_{2,l_j}$ for odd $j$. This is possible by making saddle moves that `separate' all but one of the $a_2^{l_i}$-blocks as illustrated in Figure~\ref{fig:saddlemoveonalpha}.
\begin{figure}[h]
\centering
\def\svgscale{1.6}
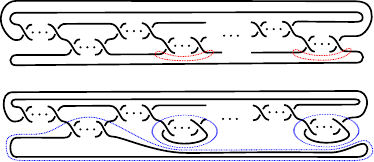
\caption{
Top: The link $\widehat{\alpha}$ with $\frac{c-1}{2}$ spheres (red) that indicate the saddle moves which yield $L$.
Bottom: The link $L$ with $\frac{c+1}{2}$ spheres (blue) that indicate how $L$ is separated into $\frac{c+3}{2}$ summands.}
\label{fig:saddlemoveonalpha}
\end{figure} Thus, we have
\[{-\s}(L)\overset{\text{\eqref{eq:sigisadd},\eqref{eq:sigbraidindex2}}}{=}\b(L)=\b(T_{2,l_{even}})+\sum_{i=0}^{\frac{c-1}{2}}\b(T_{2,l_{2i+1}})=\b(\widehat{\beta})-\frac{c-1}{2}.\]
And, therefore, we get
\[{-\s}(\widehat{\beta})\overset{\eqref{eq:sigundersaddlemove}}{\geq} {-\s}(L)-\frac{c-1}{2}=\b(\widehat{\beta})-c+1\geq\b(\widehat{\beta})-\frac{\b(\widehat{\beta})}{2}+1.\]

If $n_\beta=1$, we can assume (after possibly a conjugation) that $\beta=\Delta\alpha$ with
\[\alpha=a_1^{l_0}a_2^{l_1}a_1^{l_2}\cdots a_2^{l_{c-1}}a_1^{l_c}\]
for some even integer $c\geq2$, where 
 for all $i\in\{0,\cdots,c\}$ the $l_i$ are integers larger than or equal to $2$, by a similar argument as in the case of $n_\beta=0$. Indeed, for the first part of the statement, if a positive braid word for $\alpha$ starts with a power of $a_1$ ($a_2$) and ends with a power of $a_2$ ($a_1$), then a cyclic permutation and the braid relation $a_2^k\Delta=\Delta a_1^k$ (the braid relation $\Delta a_2^k=a_1^k\Delta$ and a cyclic permutation) allow to find a conjugate of $\beta$ for which $\alpha$ starts and ends with powers of $a_1$. And $c\geq2$ can be assumed since the case $c=0$, that is $\beta=\Delta a_1^l$ (which has closure $T_{2,2+l}$), was dealt with already.
 For the second part of the statement, we observe that $l_i=1$ for at least one $i$ allows to find a positive braid word for a conjugate of $\beta$ that starts with $\Delta^2$, which is impossible by the definition of $n_\beta$.
Note that $\b(\widehat{\beta})=\b(\widehat{\alpha})+3\geq 2c+3$.
By a similar argument as in the case of $n_\beta=0$, we get a link $L$ that is a connected sum of $\frac{c}{2}+1$ torus links of braid index two by $\frac{c}{2}$ saddle moves on $\widehat{\beta}$ that separate all of the $a_2^{l_i}$ blocks in $\alpha$ and we have
\[{-\s}(L)=\b(L)=\b(\widehat{\beta})-\frac{c}{2}.\]
And, therefore, we get
\[{-\s}(\widehat{\beta})\overset{\eqref{eq:sigundersaddlemove}}{\geq} {-\s(L)}-\frac{c}{2}=\b(\widehat{\beta})-c\geq\b(\widehat{\beta})-\frac{\b(\widehat{\beta})-3}{2}\geq\frac{\b(\widehat{\beta})}{2}+\frac{3}{2}.\]

Similar arguments work for $n_\beta=2$ or $3$. A small difference occurs: the link $L$ will be a connected sum of braid index two torus knots and $\widehat{\Delta^2a_1^{l}}$ for some positive integer $l$ (instead of just braid index two torus knots); however, ${-\s}(\widehat{\Delta^2a_1^{l}})=\b(\widehat{\Delta^2a_1^{l}})$ (see for example~\cite{Baader_14_posbraidsofmaxsig}) and so the argument remains the same.

\end{proof}
\section[Sharpness of the signature bounds]{Sharpness of the signature bounds: examples of positive $3$--braids and $4$--braids of small signature}\label{sec:examples}
In this section, we provide examples that show that the linear bounds provided in Theorem~\ref{thm:1/2for4braids} and Proposition~\ref{prop:1/2for3braids} are essentially optimal.

For every positive integer $n$, we study the following families of braids.
The  positive $3$--braids
\[\alpha_n=(a_1^2a_2^2)^{2n+1}\et{\alpha}_n'=a_2(a_1^2a_2^2)^{2n+1}\]
and the positive $4$--braids
\[\beta_n=(a_1a_3a_2^2)^{2n+1}\et{\beta}_n'=a_2(a_1a_3a_2^2)^{2n+1}.\]
We remark that the closure of ${\beta}_n'$ is a knot for all $n$.
\begin{prop}\label{prop:examples}
For all positive integers $n$, the signature of the closures of $\alpha_n$, ${\alpha}_n'$, $\beta$, and ${\beta}_n'$ are equal to the bounds provided in Theorem~\ref{thm:1/2for4braids} and Proposition~\ref{prop:1/2for3braids}, respectively:
\begin{align*}
{-\s}\left(\widehat{\alpha_n}\right)&=4n+2=\frac{8n+2}{2}+1=\frac{\b\left(\widehat{\alpha_n}\right)}{2}+1,\\
{-\s}\left(\widehat{{\alpha}_n'}\right)&=4n+3=\left\lceil\frac{8n+3}{2}+1\right\rceil
=\left\lceil\frac{\b\left(\widehat{{\alpha}_n'}\right)}{2}+1\right\rceil,\\
{-\s}\left(\widehat{{\beta}_n}\right)&=4n+1=\frac{8n+1}{2}+\frac{1}{2}=\frac{\b\left(\widehat{{\beta}_n'}\right)}{2}+\frac{1}{2},\et\\
{-\s}\left(\widehat{{\beta}_n'}\right)&=4n+2=\left\lceil\frac{8n+2}{2}+\frac{1}{2}\right\rceil
=\left\lceil\frac{\b\left(\widehat{{\beta}_n'}\right)}{2}+\frac{1}{2}\right\rceil.
\end{align*}
\end{prop}
Proposition~\ref{prop:examples} provides infinitely many examples of $3$--braids and $4$--braids of even and odd Betti number for which
the bounds of Theorem~\ref{thm:1/2for4braids} and Proposition~\ref{prop:1/2for3braids}, respectively, are realized. This proves that the bounds are optimal among all linear expressions in the Betti number with coefficients in the half integers. However, there is still room for improving the bounds from Theorem~\ref{thm:1/2for4braids} and Proposition~\ref{prop:1/2for3braids}.
For example, for closures of positive $3$--braids and $4$--braids (in fact, for all positive braid links) with Betti number $5$ or less, the Betti number equals $|\s|$ (see~\cite{Baader_14_posbraidsofmaxsig}); but this is not reflected in the bounds provided. More interestingly, what about other Betti numbers that do not occur as Betti numbers of the closures of the above families?
\begin{proof}[Proof of Proposition~\ref{prop:examples}]
First, we calculate $\s\left(\widehat{\beta_n}\right)$. The link $\widehat{\beta_n}$ is the (unoriented) boundary of the embedded annulus $A\subset\R^3$ obtained from the blackboard-framed standard link diagram for the $T_{2,2n+1}$ torus knot. In other words, $A$ is the framed knot of knot type $T_{2,2n+1}$ and framing $-4n-2$ (where the zero-framing is identified with the homological framing). For the next bit, we use the terminology of~\cite{GordonLitherland_78_OnTheSigOfALink}: the Euler number $\overline{e}(A)$ equals twice the framing of $A$ and the signature $\s(A)=\s(G_A)$ equals $1$.
Therefore, we have
\[\s\left(\widehat{\beta_n}\right)=\s(G_A)+\frac{\overline{e}(A)}{2}=1-4n-2=-4n-1.\]

To calculate the signature of $\widehat{{\beta}_n'}$, we observe that ${\beta}_n'$ is obtained from ${\beta}_n$ by adding one generator $a_2$. In other words, we can smooth one crossing in $\widehat{{\beta}_n'}$ to obtain $\widehat{{\beta}_n}$. Therefore, we have
\[{-\s}\left(\widehat{{\beta}_n'}\right)\overset{\text{\eqref{eq:sigundersaddlemove}}}
{\leq}{-\s}\left(\widehat{{\beta}_n}\right)+1=4n+2,\]
which yields ${-\s}\left(\widehat{{\beta}_n'}\right)=4n+2$ since
${-\s}\left(\widehat{{\beta}_n'}\right)\geq4n+2$ by Theorem~\ref{thm:1/2for4braids}.

To calculate $\s\left(\widehat{{\alpha}_n}\right)$ and $\s\left(\widehat{{\alpha}_n'}\right)$, we note that
$\alpha_n$ and ${\alpha}_n'$ are obtained from $\beta_n$ and ${\beta}_n'$, respectively, by replacing all generators $a_3$ with $a_1$. Following the argument in the proof of Theorem~\ref{thm:1/2for4braids}, this yields
\begin{equation}\label{eq:alphaviabeta}4n+1={-\s}\left(\widehat{{\beta}_n}\right)\geq{-\s}\left(\widehat{{\alpha}_n}\right)-1\et 4n+2={-\s}\left(\widehat{{\beta}_n'}\right)\geq {-\s}\left(\widehat{{\alpha}_n'}\right)-1.\end{equation}
This finishes the proof since the inequalities in~\eqref{eq:alphaviabeta} are equalities by Proposition~\ref{prop:1/2for3braids}.

\end{proof}
\appendix
\section{Reduction of Theorem~\ref{thm:1/8} to Theorem~\ref{thm:1/2for4braids}}
By~\cite[Reduction Lemma]{Feller_14_Sigofposbraids}, Theorem~\ref{thm:1/2for4braids} implies Theorem~\ref{thm:1/8}.
For the reader's convenience and the sake of completeness, we recall the argument.
The idea of the proof is first to smooth crossings in a given positive braid link such that a connected sum of closures of positive braids on $4$ or fewer strands remains, and then to apply Theorem~\ref{thm:1/2for4braids} to these summands.
\begin{proof}[Proof of Theorem~\ref{thm:1/8}]
We fix a positive integer $n$ and
let $\beta$ be a non-trivial positive braid in $B_n$.
Without loss of generality,
$\beta$ is not a non-trivial split union of links; in other words,
every generator $ a_{i}$ with $1\leq i\leq n-1$ is contained in $\beta$ at least once.

For $i$ in $\{1,2,3,4\}$, we denote by $\beta(i)$ the braid obtained from $\beta$ by smoothing the crossings corresponding to all but one (say the leftmost) $a_k$ for all $k$ in $\{i,i+4,i+8,i+12,\ldots \}$ as
illustrated in Figure~\ref{fig:Cutting4strands}.
\begin{figure}[h]
\centering
\def\svgscale{0.7}
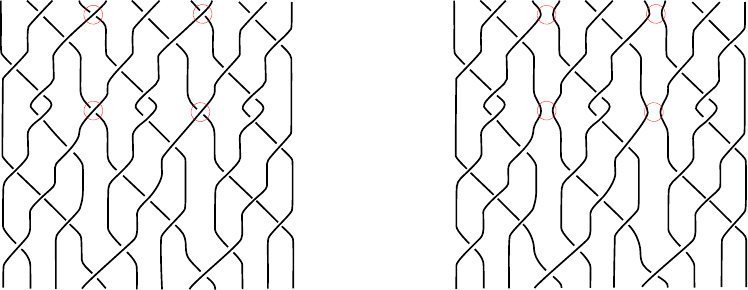
\caption{A diagram of a positive $12$--braid $\beta$ (left) with indications (red) which crossings to smooth to obtain $\beta(4)$ (right). The closure of $\beta(4)$ is a connected sum of the closures of three $4$--braids.}
\label{fig:Cutting4strands}
\end{figure}
The closure of such a $\beta(i)$ is a connected sum of closures of positive braids on $4$ or fewer strands.
Since we have $\b(\widehat{\beta})\overset{\eqref{eq:b}}{=}\sum_{k=1}^{n-1}(\sharp\{ a_k\text{ in }\beta\}-1)$, there exists an $i$ such that \begin{equation}\label{eq:bettibetai}
\b(\widehat{\beta(i)})\geq\frac{3}{4}\b(\widehat{\beta}).\end{equation} We fix such an $i$.
Let $L_1,\dots, L_l$ be closures of positive braids on at most $4$ strands such that the closure of $\beta(i)$ is the connected sum of
the $L_j$.
Therefore,
we have
\begin{align*}
{-\s}(\widehat{\beta(i)})\overset{\eqref{eq:sigisadd}}
{=} -\sum_{j=1}^{l}\s(L_j)
\overset{\text{Theorem~\ref{thm:1/2for4braids}}}{>}\sum_{j=1}^{l}\frac{\b(L_j)}{2}
{=} \frac{\b(\widehat{\beta(i)})}{2}
\overset{\eqref{eq:bettibetai}}{\geq} \frac{3\b(\widehat{\beta})}{8}.
\end{align*}
The braid $\beta(i)$ is obtained from $\beta$ by smoothing $\b(\widehat{\beta})-\b(\widehat{\beta(i)})\leq\frac{1}{4}\b(\widehat{\beta})$ crossings.
By~\eqref{eq:sigundersaddlemove}, smoothing a crossing changes the signature by at most $\pm1$;
thus, we get
\begin{align*}
{-\s}(\widehat{\beta})
\geq
-\frac{1}{4}\b(\widehat{\beta}) {-\s}(\widehat{\beta(i)})
>
-\frac{1}{4}\b(\widehat{\beta}) + \frac{3\b(\widehat{\beta})}{8}
 =
\frac{\b(\widehat{\beta})}{8}.
\end{align*}
\end{proof}
\bibliographystyle{alpha}

\def\cprime{$'$}

\end{document}